\documentclass[12pt]{amsart}

\usepackage[english]{babel}

\usepackage[pdftex,textwidth=400pt,marginratio=1:1]{geometry}

\usepackage{amsfonts}

\usepackage[dvips]{graphics}

\usepackage{amsmath}

\usepackage{amsthm}

\usepackage{amssymb}

\usepackage{eufrak}

\usepackage{cancel}

\usepackage{color}

\usepackage[curve]{xypic}

\newtheorem{theorem}{Theorem}[section]

\newtheorem{corollary}[theorem]{Corollary}

\newtheorem{proposition}[theorem]{Proposition}

\newtheorem{lemma}[theorem]{Lemma}

\theoremstyle{definition}

\newtheorem{example}[theorem]{Example}

\newtheorem{remark}[theorem]{Remark}

\theoremstyle{property}

\makeatletter
\newcommand{\contraction}[5][1ex]{%
  \mathchoice
    {\contraction@\displaystyle{#2}{#3}{#4}{#5}{#1}}%
    {\contraction@\textstyle{#2}{#3}{#4}{#5}{#1}}%
    {\contraction@\scriptstyle{#2}{#3}{#4}{#5}{#1}}%
    {\contraction@\scriptscriptstyle{#2}{#3}{#4}{#5}{#1}}}%
\newcommand{\contraction@}[6]{%
  \setbox0=\hbox{$#1#2$}%
  \setbox2=\hbox{$#1#3$}%
  \setbox4=\hbox{$#1#4$}%
  \setbox6=\hbox{$#1#5$}%
  \dimen0=\wd2%
  \advance\dimen0 by \wd6%
  \divide\dimen0 by 2%
  \advance\dimen0 by \wd4%
  \vbox{%
    \hbox to 0pt{%
      \kern \wd0%
      \kern 0.5\wd2%
      \contraction@@{\dimen0}{#6}%
      \hss}%
    \vskip 0.5ex
    \vskip\ht2}}

\newcommand{\contraction@@}[3][0.05em]{%
  \hbox{%
    \vrule width #1 height 0pt depth #3%
    \vrule width #2 height 0pt depth #1%
    \vrule width #1 height 0pt depth #3%
    \relax}}
\makeatother

\DeclareFontFamily{OT1}{rsfs}{}
\DeclareFontShape{OT1}{rsfs}{n}{it}{<-> rsfs10}{}
\DeclareMathAlphabet{\curly}{OT1}{rsfs}{n}{it}

\newcommand\LL{\mathbb L}

\renewcommand\O{\mathcal O}
\newcommand\PP{\mathbb P}

\newcommand\cP{\curly P}

\newcommand\mdot{{\scriptscriptstyle\bullet}}

\newcommand\D{\mathcal D}

\newcommand\cA{\mathcal A}

\newcommand\E{\mathbb E}

\newcommand\C{\mathbb C}

\newcommand\cC{\mathcal C}

\newcommand\FF{\mathbb F}

\newcommand\Q{\mathbb Q}
\newcommand\cQ{\mathcal Q}

\newcommand\R{\mathbb R}

\newcommand\Z{\mathbb Z}

\newcommand\cZ{\mathcal Z}

\renewcommand\k{\mathsf k}

\newcommand\Into{\ar@{^(->}[r]<-.3ex>}

\newcommand\bull{{\scriptscriptstyle\bullet}}

\newcommand\rk{\operatorname{rk}}
\newcommand\Alb{\operatorname{Alb}}

\newcommand\coker{\operatorname{coker}}
\newcommand\im{\operatorname{im}}

\newcommand\ev{\operatorname{ev}}
\newcommand\sfZ{\mathsf{Z}}
\newcommand\vir{\operatorname{vir}}
\newcommand\red{\operatorname{red}}
\newcommand\pt{\operatorname{pt}}

\newcommand\ch{\operatorname{ch}}

\newcommand\ext{\curly Ext}

\newcommand\Pic{\operatorname{Pic}}
\newcommand\AJ{\operatorname{AJ}}

\newcommand\Hilb{\operatorname{Hilb}}

\newcommand\beq[1]{\begin{equation}\label{#1}}
\newcommand\eeq{\end{equation}}
\newcommand\beqa{\begin{eqnarray*}}
\newcommand\eeqa{\end{eqnarray*}}

\DeclareRobustCommand{\SkipTocEntry}[4]{}

\begin{document}
\title[Stable pair and Seiberg-Witten invariants]{Stable pair invariants of surfaces and Seiberg-Witten invariants}
\author[M. Kool]{Martijn Kool \vspace{-5mm}}
\maketitle

\begin{abstract}



The moduli space of stable pairs on a local surface $X=K_S$ is in general non-compact. The action of $\C^*$ on the fibres of $X$ induces an action on the moduli space and the stable pair invariants of $X$ are defined by the virtual localization formula. We study the contribution to these invariants of stable pairs (scheme theoretically) supported in the zero section $S \subset X$. Sometimes there are no other contributions, e.g.~when the curve class $\beta$ is irreducible.

We relate these surface stable pair invariants to the Poincar\'e invariants of D\"urr-Kabanov-Okonek. The latter are equal to the Seiberg-Witten invariants of $S$ by work of D\"urr-Kabanov-Okonek and Chang-Kiem. We give two applications of our result. (1) For irreducible curve classes the GW/PT correspondence for $X = K_S$ implies Taubes' GW/SW correspondence for $S$. (2) When $p_g(S) = 0$, the difference of surface stable pair invariants in class $\beta$ and $K_S - \beta$ is a universal topological expression. 

\end{abstract}
\thispagestyle{empty}
\renewcommand\contentsname{\vspace{-8mm}}

\section{Introduction} \label{intro}

In \cite{PT1}, R.~Pandharipande and R.~P.~Thomas introduce stable pairs on projective 3-folds $X$ and show their moduli space is a component of the moduli space of all complexes in the bounded derived category $D^b(X)$. Formally, a stable pair $(F,s)$ on $X$ consists of a pure dimension 1 sheaf $F$ on $X$ and a section $s \in H^0(F)$ with $0$-dimensional cokernel. The moduli space of stable pairs has a perfect obstruction theory, which is symmetric in the case $X$ is Calabi-Yau. The associated invariants are known as stable pair invariants and are closely related to the Donaldson-Thomas and Gromov-Witten invariants of $X$ \cite{Bri,MNOP1, MNOP2, MOOP, PP1, PP2, PT1, PT2, Tod}. 

We consider the case where $X = K_S$ is the total space of the canonical bundle over a smooth projective surface $S$. Let $P_\chi(X,\beta)$ denote the moduli space of stable pairs $(F,s)$ on $X$ with class $\beta \in H_2(S)$ and $\chi(F)=\chi$. The space $P_\chi(X,\beta)$ carries a perfect obstruction theory but can be non-compact. Using the $\C^*$-action on the fibres of $X$ gives an induced obstruction theory on $P_\chi(X,\beta)^{\C^*}$. The components of this fixed locus are compact. For any\footnote{We denote the $\C^*$-equivariant cohomology of $X$ by $H^{*}_{\C^*}(X,\Q)$. Endowing $S$ with trivial $\C^*$-action, we then have $H^{*}_{\C^*}(X,\Q) \cong H^{*}_{\C^*}(S,\Q)$.} $\sigma_1, \ldots, \sigma_m \in H^{*}_{\C^*}(S,\Q)$ the stable pair invariants of $X$ are defined by the virtual localization formula of T.~Graber and R.~Pandharipande \cite{GP}:
\begin{equation} \label{invgeneral}
P_{\chi,\beta}(X, \tau_{\alpha_1}(\sigma_1) \cdots \tau_{\alpha_m}(\sigma_m)) := \int_{[P_\chi(X,\beta)^{\C^*}]^{\vir}} \frac{1}{e(N^{\vir})} \prod_{i=1}^{m} \tau_{\alpha_i}(\sigma_i).
\end{equation}
Here $e(N^{\vir})$ is the equivariant Euler class of the virtual normal bundle and $\tau_\alpha(\sigma)$ is the descendent insertion
\begin{equation} \label{deftau}
\tau_\alpha(\sigma) := \pi_{P*} \big( \pi_{X}^{*}(\sigma) \cap \ch_{\alpha+2}^{\C^*}(\FF) \big),
\end{equation}
where $\alpha_i \geq 0$, $\FF$ is the universal sheaf on $P_\chi(X,\beta) \times X$, and $\ch^{\C^*}$ denotes $\C^*$-equivariant Chern character. Note that these invariants are elements of $\Q[t,t^{-1}]$, where $t$ is the equivariant parameter. In this paper we will only be concerned with \emph{primary point insertions}
$$
\tau_0(\pt) := \pi_{P*} \big( \pi_{X}^{*}(\pt) \cap \ch_{2}^{\C^*}(\FF) \big),
$$
where $\pt$ denotes the (Poincar\'e dual of) the point class in $H^4(S,\Z)$.

The easiest component of $P_\chi(X,\beta)^{\C^*}$ consists of stable pairs which are scheme theoretically supported on the zero section $S \subset X$, i.e.~$P_\chi(S,\beta)$. 
Denote the Hilbert scheme of effective divisors on $S$ with class $\beta$ by $H_\beta:=\Hilb_\beta(S)$ and the universal curve by $\cC \rightarrow H_\beta$. Let $n$ be determined by $\chi = 1-h+n$, where $h$ is the arithmetic genus of curves with class $\beta$
\[
2h-2 = \beta(\beta+\k), \ \k:=c_1(\O(K_S)) \in H^2(S,\Z).
\]
Given a stable pair $[s : \O_S \rightarrow F]$ on $S$, the scheme theoretic support of $F$ is a Gorenstein curve $C \subset S$. The cokernel $Q$ of $s$ gives rise to a $0$-dimensional closed subscheme $Z \subset C$ via the surjection $\O_C \twoheadrightarrow \ext^1(Q,\O_C)$ obtained by dualizing. This provides an isomorphism \cite{PT3}
\[
P_\chi(S,\beta) \cong \Hilb^n(\cC / H_\beta),
\]
where $\Hilb^n(\cC / H_\beta)$ is the relative Hilbert scheme of $n$ points on the fibres of $\cC \rightarrow H_\beta$. In this paper we only consider contributions to \eqref{invgeneral} of the ``surface component'' $P_\chi(S,\beta)$, i.e.~
\begin{equation} \label{invS}
P_{\chi,\beta}(S, \tau_{\alpha_1}(\sigma_1) \cdots \tau_{\alpha_m}(\sigma_m)) := \int_{[P_\chi(S,\beta)]^{\vir}} \frac{1}{e(N^{\vir})} \prod_{i=1}^{m} \tau_{\alpha_i}(\sigma_i).
\end{equation}
We group these invariants together into a generating function
\begin{equation*} 
\sfZ_{\beta}^{P}(S, \tau_{\alpha_1}(\sigma_1) \cdots \tau_{\alpha_m}(\sigma_m)) := \sum_{\chi \in \Z} P_{\chi,\beta}(S, \tau_{\alpha_1}(\sigma_1) \cdots \tau_{\alpha_m}(\sigma_m)) q^\chi.
\end{equation*}
The following is our main theorem:
\begin{theorem} \label{mainthm}
For any $S, \beta$, and $m := \frac{\beta(\beta-\k)}{2}$
\[
\sfZ_{\beta}^{P}(S,\tau_0(\pt)^m) = t^m \, P_{S}(\beta) \, (q^{\frac{1}{2}} + q^{-\frac{1}{2}})^{2h-2},
\] 
where $t$ is the equivariant parameter, $2h-2 = \beta(\beta+\k)$, and $P_{S}(\beta) \in \Z$ is the numerical part of the Poincar\'e invariant $P_{S}^{+}(\beta)$ of D\"urr-Kabanov-Okonek. 
\end{theorem}

In this theorem 
$$
P_{S}^{+}(\beta) \in \Lambda^* H^1(S,\Z)^*
$$ 
are the Poincar\'e invariants of $S,\beta$ defined by M.~D\"urr, A.~Kabanov, and Ch.~Okonek \cite{DKO}. These invariants are defined in terms of a natural virtual cycle on the Hilbert scheme of curves $H_\beta$. They define a corresponding invariant $P_{S}^{-}(\beta)$ in terms of a natural virtual cycle on $H_{\k - \beta}$. We are only concerned with the numerical part (degree $b_1(S)$ in cohomology), which we denote by $P_{S}(\beta)$. D\"urr-Kabanov-Okonek conjectured that $P^{\pm}_{S}(\beta)$ are equal to the \emph{Seiberg-Witten invariants} of $S ,\beta$. Up to a purely algebraic conjecture, they prove this using their wall-crossing and blow-up formula. This algebraic conjecture was subsequently proved by H.-l.~Chang and Y.-H.~Kiem via a beautiful application of cosection localization \cite{CK}. As a corollary of the ``Poincar\'e/PT correspondence'' of Theorem \ref{mainthm} and the (much deeper!) Poincar\'e/SW correspondence of \cite{CK, DKO} we obtain:
\begin{corollary} \label{maincor}
In the notation of Theorem \ref{mainthm}
\[
\sfZ_{\beta}^{P}(S,\tau_0(\pt)^m) = t^m \, SW(\beta) \, (q^{\frac{1}{2}} + q^{-\frac{1}{2}})^{2h-2},
\] 
where $SW(\beta) \in \Z$ is the Seiberg-Witten invariant of $S,\beta$. 
\end{corollary}

We have two applications of Theorem \ref{mainthm} (and its Corollary \ref{maincor}). The first is to Gromov-Witten theory. For any $g$, let $\overline{M}_{g,m}^{\prime}(X,\beta)$ be the moduli space of stable maps with possibly disconnected domain curve and no collapsed connected components. Its $\C^*$-fixed locus $\overline{M}_{g,m}^{\prime}(S,\beta)$  has an induced perfect obstruction theory, which is the usual Gromov-Witten theory of $S$. The Gromov-Witten invariants of $X$ are defined by virtual localization
\begin{align*}
&R_{g,\beta}(X,\tau_{\alpha_1}(\sigma_1) \cdots \tau_{\alpha_m}(\sigma_m)) := \int_{[\overline{M}_{g,m}^{\prime}(S,\beta)]^{\vir}} \frac{1}{e(N^{\vir})} \prod_{i=1}^m \tau_{\alpha_i}(\sigma_i), \\
&\tau_{\alpha_i}(\sigma_i) := \psi_{i}^{\alpha_i} \ev_{i}^{*} (\sigma_i), \\
&\sfZ_{\beta}^{GW}(X,\tau_{\alpha_1}(\sigma_1) \cdots \tau_{\alpha_m}(\sigma_m)) := \sum_{g} R_{g,\beta}(X,\tau_{\alpha_1}(\sigma_1) \cdots \tau_{\alpha_m}(\sigma_m)) u^{2g-2},
\end{align*}
where $\psi_i$ are the $\psi$-classes and $\ev_i$ the evaluation maps. From Theorem \ref{mainthm} (or rather Corollary \ref{maincor}) we will deduce:
\begin{theorem} \label{GW=SW}
Fix any $S, \beta$ with $\beta$ irreducible. Let $m:=\frac{\beta(\beta-\k)}{2}$ and $2h-2 = \beta(\beta+\k)$. The GW/PT correspondence\footnote{I.e.~\cite[Conj.~3.3]{PT1} but for $X$ a non-compact Calabi-Yau 3-fold. See also \cite[Sect.~1.4]{MPT}.} for $\sfZ_{\beta}^{GW}(X,\tau_0(\pt)^m)$ and $\sfZ_{\beta}^{P}(X,\tau_0(\pt)^m)$ is equivalent to the following equality
\[
\sfZ_{\beta}^{GW}(X,\tau_0(\pt)^m) = t^m \, SW(\beta) \, (2 \sin(u/2))^{2h-2}. 
\]
In particular, setting $-q=e^{iu}$, the 
lowest order terms of $\sfZ_{\beta}^{GW}(X,\tau_0(\pt)^m)$ and $\sfZ_{\beta}^{P}(X,\tau_0(\pt)^m)$ in $u$ coincide if and only if
\[
SW(\beta) = \int_{[\overline{M}_{h,m}^{\prime}(S,\beta)]^{\vir}} \prod_{i=1}^m \tau_0(\pt).
\]
\end{theorem}
We have a similar result for any $S, \beta$ with $-K_S$ nef and $\beta$ sufficiently ample (Remark \ref{veryample}). This shows that the GW/PT correspondence implies (a very special case of) Taubes' GW/SW correspondence \cite{Tau1, Tau2}. 

The second application of Theorem \ref{mainthm} is a universal formula for the difference of stable pair invariants in class $\beta$ and $\k - \beta$. Instead of the stable pair invariants \eqref{invS}, one can define \emph{reduced} stable pair invariants of $X$ in class $\beta$
$$
P_{\chi,\beta}^{\red}(X, \tau_{\alpha_1}(\sigma_1) \cdots \tau_{\alpha_m}(\sigma_m)).
$$ 
These originate from stable pair theory on $P_\chi(X,\beta)$ by removing a trivial part of rank $p_g(S) := h^{0,2}(S)$ from the obstruction bundle. The reduced invariants coincide with the usual invariants when $p_g(S) = 0$. Reduced stable pair invariants have been studied by many people: see \cite{KT1} and references therein. Consider the surface part of these invariants for any number of point insertions\footnote{So $m$ need not be $\frac{\beta(\beta-\k)}{2}$ as in Theorems \ref{mainthm} and \ref{GW=SW}.}
$$
P_{\chi,\beta}^{\red}(S, \tau_{0}(\pt)^m).
$$ 
We recall the definition in Appendix \ref{AppA}, where we give a formula for the reduced virtual cycle (Proposition \ref{redcycle}). This formula is not used in the main body of this text, but is of independent interest. It extends a formula from \cite[Appendix]{KT2}, which was derived under the following condition 
\begin{equation} \label{strong}
H^2(L) = 0 \ \mathrm{for \ all \ line \ bundles} \  L \ \mathrm{with} \ c_1(L) = \beta.
\end{equation}
When Condition \eqref{strong} is satisfied it is shown in \cite{KT2} that $P_{\chi,\beta}^{\red}(S, \tau_{0}(\pt)^m)$ is given by a universal\footnote{This universality result is used in the recent proof of the Katz-Klemm-Vafa conjecture for all curve classes by R.~Pandharipande and R.~P.~Thomas \cite{PT4}.} function in $\beta^2$, $\beta.c_1(S)$, $c_1(S)^2$, $c_2(S)$, and certain invariants of the ring structure of $H^*(S,\Z)$. The precise statement is recalled in Theorem \ref{old} of Appendix \ref{AppA}. It is natural to ask whether universality holds for \emph{all} invariants $P_{\chi,\beta}^{\red}(S, \tau_{0}(\pt)^m)$, $P_{\chi,\beta}(S, \tau_{0}(\pt)^m)$. We show that this is \emph{not} the case (Remark \ref{failure}). The reason is as follows.  
Theorem \ref{mainthm} relates $P_{\chi,\beta}(S,\tau_0(\pt)^m)$ to Poincar\'e invariants. Using examples of \cite{DKO} we observe that Poincar\'e invariants do not satisfy universality (Examples \ref{example1}, \ref{example2}, \ref{example3} of Appendix \ref{AppB}). 

Despite failure of universality there is an interesting ``duality'' for surfaces with $p_g(S) = 0$. If $\beta$ or $\k-\beta$ satisfies Condition (\ref{strong}), then one of 
$$
P_{\chi,\beta}(S,\tau_0(\pt)^m), \ P_{\chi,\k-\beta}(S,\tau_0(\pt)^m)
$$ 
is given by a universal expression and the other is zero. These cases are covered by \cite{KT2}. The new case is when neither $\beta$ nor $\k-\beta $ satisfies (\ref{strong}). Then universality can fail for the individual invariants $P_{\chi,\beta}(S,\tau_0(\pt)^m)$, $P_{\chi,\k-\beta}(S,\tau_0(\pt)^m)$ (Examples \ref{example1}, \ref{example2}, \ref{example3} of Appendix \ref{AppB}), but their \emph{difference} satisfies a nice duality formula. Combining Theorem \ref{mainthm} and the wall-crossing formula of D\"urr-Kabanov-Okonek will lead to the following theorem:
\begin{theorem} \label{main}
Fix $S,\beta$ such that $p_g(S)=0$ and neither $\beta$ nor $\k-\beta$ satisfies Condition (\ref{strong}). If $\beta(\beta-\k)<0$, then 
\[
P_{\chi,\beta}(S,\tau_0(\pt)^m) = P_{\chi,\k-\beta}(S,\tau_0(\pt)^m)=0.
\]
If $\beta(\beta-\k) \geq 0$, then $\beta(\beta-\k)=0$, $q(S):=h^{0,1}(S)=1$, and\footnote{The fact that $\beta(\beta-\k) \geq 0$ implies $\beta(\beta-\k)=0$ and $q(S)=1$ is a non-trivial result of D\"urr-Kabanov-Okonek \cite{DKO}. This fact and its proof are recalled in Section \ref{walldual} (Proposition \ref{chi=0}). The number $[\gamma] \in \Z$ for any $\gamma \in H^2(S,\Z)$ on a surface with $q(S)=1$ is defined as follows. The class $\gamma$ determines an element $\int_S \gamma \wedge \cdot \in \Lambda^2 H^1(S,\Z)^*$. Since $q(S)=1$ we have a canonical isomorphism $\Lambda^2 H^1(S,\Z)^* \cong \Z$ induced by choosing an integral basis of $H^1(S,\Z) \subset H^1(S,\R)$ compatible with the orientation coming from the complex structure. The integer obtained in this way is denoted by $[\gamma]$.}
\begin{align*}
&P_{\chi,\beta}(S,\tau_0(\pt)^m) = P_{\chi,\k-\beta}(S,\tau_0(\pt)^m)=0, \ \mathrm{for} \ m>0 \\
&\frac{\sfZ_{\beta}^{P}(S)}{(q^{\frac{1}{2}} + q^{-\frac{1}{2}})^{2\beta^2}} - \frac{\sfZ_{\k - \beta}^{P}(S)}{(q^{\frac{1}{2}} + q^{-\frac{1}{2}})^{2(\k - \beta)^2}} = \frac{1}{2}[\beta]- \frac{1}{2}[\k-\beta], \ \mathrm{for} \ m=0.
\end{align*}
\end{theorem}
Examples of $S,\beta$ with $p_g(S)=0$, $\beta(\beta-\k) \geq 0$, and neither $\beta$ nor $\k-\beta$ satisfying Condition (\ref{strong}) are given in Remark \ref{(iii)examples} of Appendix \ref{AppB}. Such surfaces are necessarily elliptic fibrations or blow-ups thereof. The results of this paper make heavy use of the work of D\"urr-Kabanov-Okonek \cite{DKO}. For the purposes of readability, we take the opportunity to survey part of their work along the way. \\


\noindent \textbf{Acknowledgements.} I would like to thank Pierrick Bousseau, Daniel Huybrechts, Andr\'as Juh\'asz, and Ralph Klaasse for helpful discussions. Special thanks go to Richard Thomas for countless valuable comments and the anonymous referee for many suggestions on improving the exposition. This work was supported by EPSRC grant EP/G06170X/1, ``Applied derived categories''.

\section{Poincar\'e/PT correspondence}

In this section we give a formula for the virtual cycle $[\Hilb^n(\cC/H_\beta)]^{\vir}$ (Proposition \ref{nonredcycle}). We then exploit the ``product structure'' of this formula to prove Main Theorem \ref{mainthm}, Corollary \ref{maincor}, and Theorem \ref{GW=SW}.

\subsection{Virtual cycle} \label{vircycle}

Let $\cC \subset H_\beta \times S \rightarrow H_\beta$ be the universal curve over the Hilbert scheme $H_\beta = \Hilb_\beta(S)$ of effective divisors in class $\beta$. Recall from the introduction that $\Hilb^n(\cC/H_\beta) \cong P_\chi(S,\beta)$ is a component of the $\C^*$-fixed locus of the full 3-fold stable pair space $P_\chi(X,\beta)$. Also recall that $\chi = 1-h+n$, where $h$ is the genus of curves in class $\beta$. We start with the natural embedding
\[
\iota : \Hilb^n(\cC / H_\beta) \hookrightarrow S^{[n]} \times H_\beta,
\]
where $S^{[n]}$ is the Hilbert scheme of $n$ points on $S$. A point $(Z,C)$ lies in $\Hilb^n(\cC / H_\beta)$ if and only if
\[
s_{C}|_Z = 0 \in H^0(\O_{Z}(C)),
\]
where $s_C$ is the section cutting out $C \subset S$. The family version of this goes as follows. Let $\cZ \subset S^{[n]} \times S$ be the universal subscheme and let
$$
\pi : S^{[n]} \times S \times H_\beta \rightarrow S^{[n]} \times H_\beta
$$ 
denote projection. Then 
\begin{equation} \label{familytaut}
\O(\cC)^{[n]} := \pi_* \big( \O(S^{[n]} \times \cC)|_{\cZ \times H_\beta} \big)
\end{equation}
is a rank $n$ vector bundle on $S^{[n]} \times H_\beta$. It has a tautological section $\sigma$ with zero locus $\Hilb^n(\cC / H_\beta)$. This provides $\Hilb^n(\cC / H_\beta)$ with a \emph{relative} perfect obstruction theory over $H_\beta$. This construction does not provide an absolute perfect obstruction theory because $H_\beta$ can be singular. The notation \eqref{familytaut} is chosen for the following reason. Consider projections
\begin{displaymath}
\xymatrix
{
& \cZ \ar_{p}[dl] \ar^{q}[dr] & \\
S & & S^{[n]}.
}
\end{displaymath}
Then for any line bundle $L$ on $S$
\begin{equation*} \label{tautdef}
L^{[n]} := q_* p^* L
\end{equation*}
is a rank $n$ vector bundle on $S^{[n]}$ known as a tautological bundle (e.g.~see \cite{EGL}). It is not hard to see from the definitions that for any point $p = [C] \in H_\beta$
\begin{equation} \label{tautp}
\O(\cC)^{[n]} \Big|_{S^{[n]} \times \{p\}} \cong \O(C)^{[n]}.
\end{equation}

D\"urr-Kabanov-Okonek \cite{DKO} constructed a natural perfect obstruction theory on $H_\beta$ of the form 
\[
(R \pi_* \O_\cC(\cC))^\vee \rightarrow \LL_{H_\beta}.
\]
In \cite[Appendix]{KT1} this perfect obstruction theory on $H_\beta$ and the relative perfect obstruction theory on $\Hilb^n(\cC/H_\beta)$ are combined to construct an \emph{absolute} perfect obstruction theory on $\Hilb^n(\cC/H_\beta)$. See diagram (89) of \cite[Appendix]{KT1} for details. We denote the corresponding virtual cycles on $H_\beta$ and $\Hilb^n(\cC/H_\beta)$ by $[H_\beta]^{\vir}$ and $[\Hilb^n(\cC/H_\beta)]^{\vir}$. 
It is shown in \cite[Appendix]{KT1}, that $[\Hilb^n(\cC/H_\beta)]^{\vir}$ coincides with the virtual cycle induced by $\C^*$-localization of stable pair theory on $X=K_S$ to the component $\Hilb^n(\cC/H_\beta)$ of the $\C^*$-fixed locus.
Although $H_\beta$ can be singular, we still have the following:
\begin{proposition} \label{nonredcycle}
For any $S,\beta$ 
\[
\iota_* [\Hilb^n(\cC / H_\beta)]^{\vir} = (S^{[n]} \times [H_\beta]^{\vir}) . c_n\big(\O(\cC)^{[n]}\big)
\] 
and its virtual dimension is $v = \frac{\beta(\beta-\k)}{2}+n$.
\end{proposition}

For the proof of this proposition, we need the following lemma.
\begin{lemma}
Let $\pi : M \rightarrow B$ be a flat morphism of $\C$-schemes of finite type with $B$ projective. Let $E^\bull \rightarrow \LL_B$, $F^\bull \rightarrow \LL_M$ be perfect obstruction theories. Suppose that there exists a smooth projective variety $A$ and a rank $r$ vector bundle $V$ on $A \times B$ with regular\footnote{As defined in \cite[B.3.4]{Ful}.} section $s$ such that $M = s^{-1}(0) \subset A \times B$ and $\pi : M \rightarrow B$ commutes with projection $\pi_B : A \times B \rightarrow B$.  This induces a canonical relative perfect obstruction theory $G^\bull \rightarrow \LL_{M/B}$ of the form $G^\bull = \{ V^*|_M \rightarrow \pi_{A}^{*}(\Omega_A)|_M\}$. Suppose there exists an exact triangle
\begin{equation} \label{EFG}
\pi^* E^\bull \longrightarrow F^\bull \longrightarrow G^\bull.
\end{equation}
Denote inclusion by $\iota : M \hookrightarrow A \times B$. Then
\begin{equation} \label{cycle}
\iota_* [M]^{\vir} = (A \times [B]^{\vir}) . c_r(V).
\end{equation}
\end{lemma}
\begin{proof}
The content of the lemma is formula \eqref{cycle}. For any perfect obstruction theory $F^\bull \rightarrow \LL_M$ with $M$ projective, the following formula holds \cite[Thm.~4.6]{Sie} (see also \cite{Pid})
\begin{equation} \label{Siebert}
[M]^{\vir} = \Big\{ s_\bull(F^{\bull \vee}) \, c_F(M) \Big\}_{v}.
\end{equation}
Here $s_\bull(\cdot)$ is the total Segre class, $v$ is the virtual dimension of $M$, and $c_F(M)$ is Fulton's canonical class which is defined as follows. Take any embedding $M \subset \cA$ into a smooth variety $\cA$, then
\[
c_F(M) := c_\bull(T_\cA|_M) \, s_\bull(C_{M/\cA}),
\]
where $C_{M/\cA}$ is the normal cone of $M \subset \cA$. This definition is independent of choice of embedding \cite[Ex.~4.2.6]{Ful}. Take an embedding $B \subset C$ into a smooth variety and consider
\[
M \subset A \times B \subset A \times C=: \cA.
\]
By \eqref{EFG} we have
\[
s_\bull(F^{\bull \vee}) = \pi^* (s_\bull(E^{\bull \vee})) \, \frac{c_\bull(V|_M)}{\pi_{A}^{*}(c_\bull(T_A))|_M}.
\]
Since $M \subset A \times B$ is cut out by a regular section of $V$, we have 
$$
C_{M / A \times B} \cong N_{M / A \times B} \cong V|_M.
$$ 
Consider the following short exact sequence of cones 
\[
N_{M / A \times B} \longrightarrow C_{M / A \times C} \longrightarrow C_{A \times B / A \times C}|_M.
\]
We deduce
\[
c_F(M) = \pi_{A}^{*} (c_\bull(T_A))|_M \, \pi^*(c_\bull(T_C|_B)) \, \frac{\pi^* s_\bull(C_{B / C})}{c_\bull(V|_M)}.
\]
Formula (\ref{Siebert}) therefore implies
\[
[M]^{\vir} = \Bigg\{ \pi^*\Big(s_\bull(E^{\bull \vee}) \, c_\bull(T_C|_B) \, s_\bull(C_{B / C}) \Big) \Bigg\}_{v} = \pi^*[B]^{\vir},
\]
where the second equality follows from applying (\ref{Siebert}) to $E^{\bull} \rightarrow \LL_B$. The projection formula gives
\[
\iota_* [M]^{\vir} = (A \times [B]^{\vir}) . \iota_*[M].
\]
Since $M \subset A \times B$ is cut out by a regular section of $V$, we have $\iota_*[M] = c_r(V)$ \cite[Prop.~14.1]{Ful} and the proposition is proved.
\end{proof}

\begin{proof} [Proof of Proposition \ref{nonredcycle}]
Diagram (89) of \cite[Appendix]{KT1} provides the required exact triangle. It is left to show $\Hilb^n(\cC / H_\beta) \rightarrow H_\beta$ is flat and the tautological section $\sigma$ of $\O(\cC)^{[n]}$ is regular. The fibre of the morphism $\Hilb^n(\cC / H_\beta) \rightarrow H_\beta$ over $C \in H_\beta$ is $C^{[n]}$, i.e.~the Hilbert scheme of $n$ points on the effective divisor $C$. The scheme $C^{[n]}$ is cut out by a tautological section of $L^{[n]}$ where $L:=\O(C)$. Moreover, $C^{[n]} \subset S^{[n]}$ has codimension $n$ (see \cite[Footnote 18]{KT1}, which uses \cite{AIK, Iar}). Therefore $\sigma|_{S^{[n]} \times \{C\}}$ is regular for all $C \in H_\beta$. From this one can deduce that $\Hilb^n(\cC / H_\beta) \rightarrow H_\beta$ is flat and $\sigma$ is regular.
\end{proof}

\subsection{Relation to Poincar\'e invariants}

In Section \ref{intro} we introduced the stable pair invariants \eqref{invgeneral}
$$
P_{\chi,\beta}(X, \tau_{\alpha_1}(\sigma_1) \cdots \tau_{\alpha_m}(\sigma_m))
$$
and the contribution to these invariants of the component $P_\chi(S,\beta) \cong \Hilb^n(\cC/H_\beta)$ of the $\C^*$-fixed locus 
$$
P_{\chi,\beta}(S, \tau_{\alpha_1}(\sigma_1) \cdots \tau_{\alpha_m}(\sigma_m)).
$$
We only consider the case of primary point insertions
$$
P_{\chi,\beta}(S, \tau_{0}(\pt)^m) = \int_{[P_\chi(S,\beta)]^{\vir}} \frac{1}{e(N^{\vir})} \, \tau_{0}(\pt)^m.
$$

In the case $n=0$, $\Hilb^n(\cC/H_\beta) \cong H_\beta$ and $[H_\beta]^{\vir}$ was introduced many years ago by D\"urr-Kabanov-Okonek \cite[Def.~3.1]{DKO}. They used this virtual cycle to define \emph{Poincar\'e invariants}. We recall their definition. Consider the two Abel-Jacobi maps 
\begin{align*}
&\AJ^+ : H_\beta \rightarrow \Pic^\beta(S), \\ 
&\AJ^- : H_{\k-\beta} \rightarrow \Pic^{\k-\beta}(S) \cong \Pic^\beta(S), 
\end{align*}
where $\Pic^{\k-\beta}(S) \cong \Pic^\beta(S)$, $L \mapsto L^* \otimes K_S$. Then 
the Poincar\'e invariants are
\begin{align} \label{Poininv} 
P^{+}_{S}(\beta) := \AJ^{+}_{*} \Big( \sum_i c_1( \O(\cC)|_{H_\beta \times \{\pt\}} )^i \cap [H_\beta]^{\vir} \Big), 
\end{align}
\begin{align*}
P^{-}_{S}(\beta) :=(-1)^{\chi(\O_S) + \frac{\beta(\beta-\k)}{2}} \AJ^{-}_{*} \Big( \sum_i (-1)^i c_1( \O(\cC)|_{H_{\k-\beta} \times \{\pt\}} )^i \cap [H_{\k-\beta}]^{\vir} \Big).
\end{align*}
In the first line, $\cC$ denotes the universal curve over $H_\beta$ and in the second line, the universal curve over $H_{\k-\beta}$. Note that\footnote{From the construction the Poincar\'e invariants take values in homology $H_*(\Pic^\beta(S)) \cong \Lambda^* H^1(S,\Z)$. We use Poincar\'e duality so the invariants take values in cohomology $\Lambda^* H^1(S,\Z)^*$.} $P^{\pm}_{S}(\beta) \in \Lambda^* H^1(S,\Z)^*$. We write the (numerical) degree $2q(S)$ part of $P^{+}_{S}(\beta) \in \Lambda^* H^1(S,\Z)^*$ by
$$
P_{S}(\beta) \in \Z.
$$
The product structure of the virtual cycle of Proposition \ref{nonredcycle} leads to Main Theorem \ref{mainthm} of the introduction:
\begin{proof}[Proof of Theorem \ref{mainthm}] 
We want to calculate the invariant
\begin{equation} \label{start}
P_{\chi,\beta}(S,\tau_0(\pt)^m) := \frac{1}{e(N^{\vir})} \ \tau_0(\pt)^m \cap [\Hilb^n(\cC/H_\beta)]^{\vir}, 
\end{equation}
where $\Hilb^n(\cC/H_\beta) \cong P_{\chi}(S,\beta)$, and $\chi$ and $n$ are related by $\chi = 1-h+n$ (Section \ref{intro}). Let $\varpi : \Hilb^n(\cC/H_\beta) \rightarrow H_\beta$ denote projection, then we claim 
\begin{equation} \label{tau0}
\tau_0(\pt) = \varpi^* c_1(\O(\cC)|_{H_\beta \times \{\pt\}}).
\end{equation}
The proof can be found in \cite[Proof Cor.~4.2]{KT2}, but we quickly reproduce it here. Consider the Cartesian diagram
\begin{displaymath}
\xymatrix
{
S & \Hilb^n(\cC/H_\beta) \times S \ar^>>>>{\pi_P}[r] \ar[d] \ar_<<<{\pi_S}[l] & \Hilb^n(\cC/H_\beta) \ar[d] \\
& H_\beta \times S \ar[r] & H_\beta.
}
\end{displaymath}
By the definition \eqref{deftau}, $\tau_0(\pt) = \pi_{P*}(\pi_{S}^{*}[\pt] \cdot c_1(\FF))$, where $\FF$ is the universal sheaf on $\Hilb^n(\cC/H_\beta) \times S$. Hence \eqref{tau0} follows from the fact that $c_1(\FF)$ is the pull-back of $c_1(\O(\cC))$ from $H_\beta \times S$ and going around the Cartesian diagram.

In order to calculate $e(N^{\vir})$, we use a formula for the $\C^*$-equivariant $K$-theory class of $N^{\vir}$ from \cite{KT2}. Consider the projections
\begin{displaymath}
\xymatrix
{
& S^{[n]} \times H_\beta \ar_{p_1}[dl] \ar^{p_2}[dr] & \\
S^{[n]} & & H_\beta.
}
\end{displaymath} 
Then \cite[Eqn.~(12)]{KT2} reads
\begin{equation} \label{NvirfromKT2}
[N^{\vir}] = \big[ \big( \O(\cC)^{[n]} \big)^* - p_{1}^{*} \, \Omega_{S^{[n]}} - p_{2}^{*} \big( R\pi_{*}\O_{\cC}(\cC) \big)^\vee \big]\Big|_{\Hilb^n(\cC/H_\beta)} \otimes \mathfrak{t},
\end{equation}
where $\mathfrak{t}$ is the irreducible representation of $\C^*$ of weight 1. Recall from \eqref{familytaut} that $\O(\cC)^{[n]}$ is a vector bundle on $S^{[n]} \times H_\beta$, $R\pi_{*}\O_{\cC}(\cC)$ is a complex on $H_\beta$, and $\pi$ denotes projection $H_\beta \times S \rightarrow H_\beta$. By pushing forward along the inclusion $\iota : \Hilb^n(\cC/H_\beta) \hookrightarrow S^{[n]} \times H_\beta$ and using \eqref{start}, \eqref{tau0}, \eqref{NvirfromKT2}, we see that $P_{\chi,\beta}(S,\tau_0(\pt)^m)$ equals
$$
\frac{e(p_{1}^{*} \, \Omega_{S^{[n]}} \otimes \mathfrak{t}) \cdot \, e\big( p_{2}^{*} \big(R\pi_{*}\O_{\cC}(\cC)\big)^\vee \otimes \mathfrak{t}\big)}{e\big( \big( \O(\cC)^{[n]} \big)^* \otimes \mathfrak{t} \big)} \cdot \varpi^* c_1(\O(\cC)|_{H_\beta \times \{\pt\}})^m \, \cap \, \iota_* [\Hilb^n(\cC/H_\beta)]^{\vir}.
$$
Next we want to use the formula for $\iota_* [\Hilb^n(\cC/H_\beta)]^{\vir}$ from Proposition \ref{nonredcycle}. Recall from the assumptions of the theorem that $m:=\frac{\beta(\beta-\k)}{2}$. Since the virtual dimension of $[H_\beta]^{\vir}$ is also $\frac{\beta(\beta-\k)}{2}$, the cycle
\begin{equation*} \label{eqn1}
c_1(\O(\cC)|_{H_\beta \times \{\pt\}})^m \cap [H_\beta]^{\vir}
\end{equation*}
is 0-dimensional and can be written as $\sum_i \mu_i \, p_i$, where $\mu_i$ are integers and $p_i = [C_i] \in H_\beta$ are points. Then 
$$
P_S(\beta) = \sum_i \mu_i
$$
by definition of the Poincar\'e invariants \eqref{Poininv}. Therefore $P_{\chi,\beta}(S,\tau_0(\pt)^m)$ equals
\begin{equation*} \label{long}
\sum_i \mu_i \int_{S^{[n]}} \frac{e(p_{1}^{*} \, \Omega_{S^{[n]}} \otimes \mathfrak{t}) \cdot \, e\big( p_{2}^{*} \big(R\pi_{*}\O_{\cC}(\cC)\big)^\vee \otimes \mathfrak{t}\big)}{e\big( \big( \O(\cC)^{[n]} \big)^* \otimes \mathfrak{t} \big)} \ c_n\big(\O(\cC)^{[n]}\big) \Bigg|_{S^{[n]} \times \{p_i\}}.
\end{equation*}
In order to go from equivariant Euler classes to Chern classes we use the following formula (e.g.~\cite[Eqn.~(16)]{KT2}). For any complex $E$ of rank $r$
\begin{equation} \label{EulerChern}
e(E \otimes \mathfrak{t}) = t^r c_{-1/t}(E^\vee),
\end{equation}
where $c_{x}(E) = 1 + c_1(E)x + c_2(E)x^2+ \cdots$ is the total Chern class and $t := c_1(\mathfrak{t})$ is the equivariant parameter. Define $L_i := \O_S(C_i)$, where $p_i = [C_i] \in H_\beta$ was introduced earlier in the proof. Then \eqref{tautp} implies
\begin{equation} \label{tautpt}
\O(\cC)^{[n]} \Big|_{S^{[n]} \times \{p_i\} } \cong L_{i}^{[n]}.
\end{equation}
Similarly 
\begin{equation} \label{RGammapt}
p_{2}^{*} \, R\pi_{*}\O_{\cC}(\cC) \Big|_{S^{[n]} \times \{p_i\} }  \cong R\Gamma(\O_{C_i}(C_i)) \otimes \O.
\end{equation}
Using \eqref{EulerChern}, \eqref{tautpt}, \eqref{RGammapt} shows that $P_{\chi, \beta}(S,\tau_0(\pt)^m)$ equals
\begin{align}
&\sum_i \mu_i \int_{S^{[n]}} \frac{t^{2n} \, c_{-1/t}(T_{S^{[n]}}) \cdot t^{1-h+\beta^2} \, c_{-1/t}(R\Gamma(\O_{C_i}(C_i)) \otimes \O)}{ t^{n} \, c_{-1/t}(L_{i}^{[n]})} \, c_n(L_{i}^{[n]}) \nonumber \\
&= \sum_i \mu_i \int_{S^{[n]}}  t^{n+m} \Bigg( \frac{-1}{t} \Bigg)^n \frac{c_{\mdot}(T_{S^{[n]}})}{c_{\mdot}(L_{i}^{[n]})} \, c_n(L_{i}^{[n]}) \nonumber \\
&= (-1)^n \, t^m \, \sum_i \mu_i \, \int_{S^{[n]}} \frac{c_{\mdot}(T_{S^{[n]}})}{c_{\mdot}(L^{[n]}_{i})} \, c_n(L^{[n]}_{i}), \label{integrals}
\end{align}
where the second equality uses $m := \frac{\beta(\beta-\k)}{2}$ and the factor $(-1/t)^n$ arises from the fact that $c_n(L_{i}^{[n]})$ has degree $n$ and $S^{[n]}$ has dimension $2n$.

By \cite{EGL}, for each $n$ there exists a universal polynomial $P_n(x_1,x_2,x_3,x_4)$ such that for all $i$ we have
$$
P_n(c_1(L_i)^2, c_1(L_i).\k, \k^2, c_2(S)) = \int_{S^{[n]}} c_n(L_{i}^{[n]}) \, \frac{c_{\mdot}(T_{S^{[n]}})}{c_{\mdot}(L_{i}^{[n]})}.
$$
Since $c_1(L_i) = \beta$ for all $i$, all these integrals are the same. Using $P_S(\beta) = \sum_i \mu_i$, formula \eqref{integrals} becomes
\begin{equation} \label{finalint}
P_{\chi,\beta}(S,\tau_0(\pt)^m) = (-1)^n \, t^m \, P_S(\beta) \, \int_{S^{[n]}} c_n(L^{[n]})  \, \frac{c_{\mdot}(T_{S^{[n]}})}{c_{\mdot}(L^{[n]})},
\end{equation}
where $L:=L_i$ for arbitrary choice of $i$. 

For \emph{any} $S,L$, the integral in \eqref{finalint} is given by $P_n(c_1(L)^2, c_1(L).\k, \k^2, c_2(S))$. For any $S,L$ with the additional property that $L$ is globally generated, we can compute the integral in \eqref{finalint}. If $L$ is globally generated, we can write $L = \O(C)$ for a smooth curve $C \subset S$. Then the Hilbert scheme $C^{[n]}$ of $n$ points on $C$ is cut out smoothly and transversally by a tautological section of $L^{[n]}$. Hence
$$
\int_{S^{[n]}} c_n(L^{[n]})  \, \frac{c_{\mdot}(T_{S^{[n]}})}{c_{\mdot}(L^{[n]})} = \int_{C^{[n]}} c_n(T_{C^{[n]}}) = e(C^{[n]}).
$$
These Euler characteristics are given by the well-known expression
\[
\sum_{n=0}^{\infty} e(C^{[n]}) q^n = (1-q)^{2g-2},
\]
where $g$ is the genus of $C$. We conclude that
\begin{align} 
\begin{split} \label{eqs}
&P_n(c_1(L)^2, c_1(L).\k, \k^2, c_2(S)) = \int_{S^{[n]}} c_n(L^{[n]}) \,  \frac{c_{\mdot}(T_{S^{[n]}})}{c_{\mdot}(L^{[n]})} = (-1)^n \binom{2g-2}{n}, \\
&\textrm{where \ } 2g-2 = c_1(L)^2 + c_1(L).\k.
\end{split}
\end{align}
Since \eqref{eqs} holds for any $S,L$ with $L$ globally generated and $P_n$ is polynomial, it holds for \emph{any} $S,L$. The theorem follows by combining \eqref{finalint} and \eqref{eqs}.
\end{proof}

\subsection{Application to Seiberg-Witten invariants}

D\"urr-Kabanov-Okonek conjectured that Poincar\'e invariants \eqref{Poininv} are equal to Seiberg-Witten invariants from 4-manifold theory \cite[Conj.~5.3]{DKO}. Using a wall-crossing formula and blow-up formula for $P^{\pm}_{S}(\beta)$, they reduced their conjecture to a purely algebraic statement about $H_{\k}$, which was proved by Chang-Kiem \cite{CK}.
By these (non-trivial!) results, we can write the degree $2q(S)$ part of $P^{+}_{S}(\beta)$ as
$$
P_S(\beta) = SW(\beta) \in \Z,
$$
where $SW(\beta)$ are the original Seiberg-Witten invariant of $S, \beta$ (see \cite{Wit, Moo}). Combining the Poincar\'e/PT correspondence of Theorem \ref{mainthm} with the (much deeper!) Poincar\'e/SW correspondence of \cite{DKO, CK} gives Corollary \ref{maincor}. An application of this corollary is that for $S,\beta$ with $\beta$ \emph{irreducible} and $m = \frac{\beta(\beta-\k)}{2}$ point insertions the GW/PT correspondence encodes (a very special case of) Taubes' GW/SW correspondence \cite{Tau1, Tau2}. This is the content of Theorem \ref{GW=SW} of the introduction:
\begin{proof}[Proof of Theorem \ref{GW=SW}]
Since $\beta$ is irreducible, $P_\chi(X,\beta)^{\C^*} \cong P_\chi(S,\beta)$ for all $\chi$. Hence $P_\beta(X,\tau_0(\pt)^m) = P_\beta(S,\tau_0(\pt)^m)$ and the result follows from Theorem \ref{mainthm}. Note that the equivariant parameter $t$ of the leading term of both generating functions match by \cite[Lem.~3.3]{KT1}.
\end{proof}

\begin{remark} \label{veryample}
The following is a variation on Theorem \ref{GW=SW}. Fix any $S, \beta$ such that $-K_S$ is nef and $\beta$ is sufficiently ample\footnote{I.e.~$\beta$ such that $h \geq 1$ and $\beta$ is $(4h-3)$-very ample \cite[Prop.~5.1]{KT2}.}. Assume the GW/PT correspondence\footnote{The GW/PT correspondence has been proved in many cases \cite{MOOP, MPT, PP1, PP2}.} holds for $\sfZ_{\beta}^{GW}(X,\tau_0(\pt)^m)$, $\sfZ_{\beta}^{P}(X,\tau_0(\pt)^m)$. Also assume that the BPS spectrum of $X$ is finite\footnote{I.e.~after writing $\sfZ^{GW}_{\beta}(X,\tau_0(\pt)^m)$ in BPS form \cite{GV1, GV2}, \cite[Eqn.~(3.13)]{PT1}, we assume there are only finitely many nonzero $n_{g,\beta'}$.}. Then
\begin{align*}
&\sfZ_{\beta}^{GW}(X,\tau_0(\pt)^m) = t^m \, SW(\beta) \, (2 \sin(u/2))^{2h-2}, \\
&SW(\beta) = \int_{[\overline{M}_{h,m}^{\prime}(S,\beta)]^{\vir}} \prod_{i=1}^m \mathrm{ev}_{i}^{*}[\pt].
\end{align*}

The proof goes as follows. Since $h \geq 1$ and the BPS spectrum is assumed finite, applying the coordinate transformation $-q = e^{iu}$ to $\sfZ^{GW}_{\beta}(X,\tau_0(\pt)^m)$ gives a Laurent \emph{polynomial} in $q$. Moreover, it is symmetric under $q \leftrightarrow q^{-1}$, so of the form
\begin{equation} \label{Laurent}
a_b q^{-b} + a_{b-1} q^{-(b-1)} + \cdots + a_{b-1} q^{b-1} + a_b q^b,
\end{equation}
for some $b \geq 0$. By \cite[Prop.~5.1]{KT1}, we have $P_\chi(X,\beta)^{\C^*} \cong P_\chi(S,\beta)$ for all $\chi \leq h-1$. 
Combining this with Theorem \ref{mainthm} and (\ref{Laurent}) gives the result. \hfill $\oslash$
\end{remark}

\begin{remark}
One can speculate that for \emph{any} algebraic $S,\beta$, Taubes' GW/SW correspondence follows from the GW/PT correspondence. This requires dealing with other components of $P_\chi(X,\beta)^{\C^*}$. Conversely, one can try to derive cases of the GW/PT correspondence for $X = K_S$ from Taubes' GW/SW correspondence as is done in Theorem \ref{GW=SW}. These are interesting questions for future research. \hfill $\oslash$
\end{remark}

\section{Wall-crossing and duality} \label{walldual}

In this section we study the stable pair invariants $P_{\chi,\beta}(S,\tau_0(\pt)^m)$ for any $m$ and any surface $S$ with $p_g(S) = 0$. The results of \cite{KT2} (recalled in Theorem \ref{old} of Appendix \ref{AppA}) suggest that these invariants are \emph{always} given by universal functions in the topological numbers $\beta^2$, $\beta.c_1(S)$, $c_1(S)^2$, $c_2(S)$ and certain invariants of the ring structure of $H^1(S,\Z)$. In Appendix \ref{AppB} we show that this is \emph{not} the case. The reason is that $P_{\chi,\beta}(S,\tau_0(\pt)^m)$ is related to a Poincar\'e invariant by Main Theorem \ref{mainthm} and it is easy to cook up surfaces $S$ with $p_g(S) = 0$ whose Poincar\'e invariants are \emph{not} given by universal functions (Examples \ref{example1}, \ref{example2}, \ref{example3} of Appendix \ref{AppB}).  

However, D\"urr-Kabanov-Okonek prove that when $p_g(S) = 0$ the difference of Poincar\'e invariants in class $\beta$ and $\k - \beta$ satisfies a universal formula. Combining their formula with Main Theorem \ref{mainthm} gives an expression for the difference of $P_{\chi,\beta}(S,\tau_0(\pt)^m)$ and $P_{\chi,\k - \beta}(S,\tau_0(\pt)^m)$. This is Theorem \ref{main} of the introduction and the second application of Main Theorem \ref{mainthm}.

\subsection{D\"urr-Kabanov-Okonek's wall-crossing}

We recall the wall-crossing formula for Poincar\'e invariants \cite[Thm.~3.16]{DKO}. Since we use this formula to establish Theorem \ref{main}, and for the sake of completeness, we recall D\"urr-Kabanov-Okonek's interesting argument. Moreover, their results lead to a nice observation about the \emph{reduced} virtual cycle for stable pairs, which is of independent interest (Proposition \ref{redcycle} in Appendix \ref{AppA}). The results and arguments presented in this section come \emph{entirely} from their paper \cite{DKO}.

The following is contained in \cite[Lem.~2.17]{DKO} and its proof (see also \cite[Cor.~3.15]{DKO}).
\begin{proposition}[D\"urr-Kabanov-Okonek] \label{chi=0}
Let $S$ be any surface. Suppose that $\beta$ satisfies the following conditions:
\begin{enumerate}
\item[(i)] For any \emph{effective} $L \in \Pic^\beta(S)$ with $c_1(L) = \beta$ we have $H^2(L) = 0$. Note: this is automatic when $p_g(S) = 0$.
\item[(ii)] $\beta(\beta - \k) \geq 0$.
\item[(iii)] $H_\beta$ and $H_{\k -\beta}$ are both non-empty.
\end{enumerate}
Then $\beta(\beta-\k)=0$ and $\chi(\O_S) = 0$. Note: $\chi(\O_S)=0$ is equivalent to $q(S)=1$ when $p_g(S) = 0$. 
\end{proposition}
\begin{proof}
The result follows by showing 
$$
\frac{\beta(\beta-\k)}{2} + \chi(\O_S) = 0 \ \textrm{and} \ \chi(\O_S) \geq 0.
$$
Let $p : \Pic^\beta(S) \times S \rightarrow \Pic^\beta(S)$ denote projection and let $\cP$ be a choice of Poincar\'e bundle on $\Pic^\beta(S) \times S$.

Condition (i) is equivalent to the statement that the images (Brill-Noether loci) of the two maps $H_{\k-\beta} \rightarrow \Pic^{\k-\beta}(S) \cong \Pic^\beta(S)$ and $H_\beta \rightarrow \Pic^\beta(S)$ are disjoint. In other words, their complements $U$ and $V$ satisfy $\Pic^\beta(S) = U \cup V$. Moreover, for any $L \in \Pic^\beta(S)$, we have $H^2(L) = 0$ when $L \in U$ and $H^0(L) = 0$ when $L \in V$. In other words 
\begin{equation*} \label{vanishing}
R^2 p_* \cP |_U = 0, \ R^0 p_* \cP |_V = 0.
\end{equation*}
This implies 
$$
\rk Rp_* \cP = \frac{\beta(\beta-\k)}{2} + \chi(\O_S) \leq 0.
$$

Since $H_\beta$ and $H_{\k -\beta}$ are both non-empty (Condition (iii)), $S$ cannot be rational because otherwise we get a section of $K_S$. Moreover, $S$ cannot be ruled: for $F$ the class of a fibre either $\beta.[F] < 0$ in which case $H_\beta = \varnothing$ or $\beta.[F] \geq 0$ in which case $(\k-\beta).[F] < 0$ so $H_{\k-\beta} = \varnothing$. Similarly, $S$ cannot be the blow-up of a ruled surface. We conclude that the Kodaira dimension of $S$ is $\geq 0$. Therefore $\chi(\O_S) \geq 0$ and 
\begin{equation*}
\frac{\beta(\beta-\k)}{2} + \chi(\O_S) \geq 0. \qedhere
\end{equation*}
\end{proof}

\begin{theorem}[D\"urr-Kabanov-Okonek] \label{wall-crossing}
Let $S$ be a surface with $p_g(S) = 0$. Let $\cP$ be a choice of normalized Poincar\'e bundle on $\Pic^\beta(S)$, i.e.~$\cP|_{\Pic^\beta(S) \times \{\pt\}} \cong \O$. Denote projection by $p : \Pic^\beta(S) \times S \rightarrow \Pic^\beta(S)$. Then
\[
P^{+}_{S}(\beta) - P^{-}_{S}(\beta) = \sum_{i \geq 1 - \chi(\beta)} s_i(p_! \cP),
\] 
where $1-\chi(\beta) = q(S) - \frac{\beta(\beta-\k)}{2}$.
\end{theorem}
\begin{proof}
We first note that $\beta$ satisfies Condition \eqref{strong} of the introduction if and only if $H_{\k - \beta} = \varnothing$. Indeed if $\beta$ satisfies Condition \eqref{strong} we clearly have $H_{\k - \beta} = 0$. Conversely $H_{\k - \beta} = \PP(R^2 p_* \cP)$ by \cite[Lem.~2.15]{DKO}, so if $H_{\k - \beta} = \varnothing$ we have $R^2 p_* \cP = 0$ and hence $\beta$ satisfies Condition \eqref{strong} by cohomology and base change. Similarly $\k - \beta$ satisfies Condition \eqref{strong} if and only if $H_\beta = \varnothing$ (using $H_{\beta} = \PP(R^2 p_* \cP^*(K_S))$ \cite[Lem.~2.15]{DKO}).

The rest of the proof of \cite{DKO} runs as follows. If $\frac{\beta(\beta-\k)}{2} < 0$, then the virtual dimension of $H_\beta$ and $H_{\k-\beta}$ are negative so the LHS is zero. Moreover the RHS is zero because of degree reasons ($\Pic^\beta(S)$ has dimension $q(S)$).
For the remainder of the proof assume $\frac{\beta(\beta-\k)}{2} \geq 0$. 

Let $\cP$ be a choice of Poincar\'e bundle on $\Pic^\beta(S) \times S$ and let 
$$
p : \Pic^\beta(S) \times S \rightarrow \Pic^\beta(S)
$$ 
denote projection. 
In Appendix \ref{AppA} we describe a construction, which embeds $H_\beta$ into a smooth ambient space in a natural way. For \emph{sufficiently ample} divisor $A \subset S$ define $\gamma := [A] + \beta$ and let $\cQ$ be a choice of Poincar\'e bundle on $\Pic^\gamma(S) \times S$. Again we denote projection by $p : \Pic^\gamma(S) \times S \rightarrow \Pic^\gamma(S)$. By sufficient ampleness of $A$, the Abel-Jacobi map
$$
\AJ : H_\gamma \longrightarrow \Pic^\gamma(S)
$$
is a projective bundle and $H_\gamma \cong \PP(p_* \cQ)$. Moreover we can embed $H_\beta \hookrightarrow H_\gamma$ by adding the divisor $A$. There exists a natural sheaf $F$ on $H_\gamma$ with tautological section cutting out $H_\beta \hookrightarrow H_\gamma$. Since $p_g(S) = 0$, the sheaf $F$ is a vector bundle on a Zariski open neighbourhood of $H_\beta$. See Appendix \ref{AppA} for the details. 
Let $r$ be the rank of $F$ and let $h := c_1(\O(1))$ on $\PP(p_* \cQ)$. If $H_\beta \neq \varnothing$, then 
$$
\iota_* [H_{\beta}]^{\vir} = c_r(F)
$$ 
on $H_\gamma$ (Proposition \ref{redcycle} of Appendix \ref{AppA} for $n=0$) and 
\begin{equation} \label{formulaAJ}
\AJ_* \big( c_1(\O(\cC)|_{H_\beta \times \{\pt\}})^i \cap [H_\beta]^{\vir} \big) = \AJ_*(c_r(F) h^i).
\end{equation}
A similar formula holds for $[H_{\k-\beta}]^{\vir}$ when $H_{\k-\beta} \neq \varnothing$. Moreover by \cite[Prop.~2.18]{DKO} (or \cite[Lem.~4.3]{KT2})
\begin{equation} \label{Segre}
\AJ_*(c_r(F) h^i)  =  s_{i-\chi(\beta)+1}(\tau_{\leq 1} p_! \cP),
\end{equation}
where $\chi(\beta)$ denotes the holomorphic Euler characteristic of $\beta$. Equation \eqref{Segre} also holds when $H_\beta = \varnothing$. 

If $\beta$ satisfies Condition (\ref{strong}) (i.e.~$H_{\k-\beta} = \varnothing$), then $F$ is a vector bundle on $H_\gamma$, and $s_i(p_! \cP) = s_i(\tau_{\leq 1} p_! \cP)$. The formula follows from \eqref{Segre} and \eqref{formulaAJ}. If $\k-\beta$ satisfies Condition (\ref{strong}) (i.e.~$H_\beta = \varnothing$), then the formula follows similarly using Serre duality $Rp_*\cP^*(K_S) \cong (Rp_* \cP [2])^\vee$. 

We are left with the case where neither $\beta$ nor $\k-\beta$ satisfies Condition (\ref{strong}), i.e.~$H_\beta$ and $H_{\k - \beta}$ are both non-empty. The wall-crossing formula is equivalent to
\[
\sum_{i \geq 1 - \chi(\beta)} \Big(s_i(\tau_{\leq 1} p_! \cP) + (-1)^i s_i( \tau_{\leq 1} p_! \cP^*(K_S)) \Big) = \sum_{i \geq 1 - \chi(\beta)} s_i(p_! \cP).
\] 

By Proposition \ref{chi=0}, $\beta(\beta-\k) \geq 0$ and $H_\beta$, $H_{\k - \beta}$ are both non-empty implies $\chi(\O_S) = 0$ and $\beta(\beta-\k)=0$. Since $p_g(S)=0$, we have $q(S)=1$. Since $s_1(\tau_{\leq 1} p_! \cP) = c_1(R^1 p_* \cP) - c_1(p_* \cP)$, it suffices to show 
$$
s_1(\tau_{\leq 1} p_! \cP^*(K_S)) = c_1(R^2 p_* \cP).
$$ 

Take a locally free resolution $[E^0 \stackrel{d^0}{\rightarrow} E^1 \stackrel{d^1}{\rightarrow} E^2]$ of $R p_* \cP$. Then Serre duality $Rp_*\cP^*(K_S) \cong (Rp_* \cP [2])^\vee$ implies
\begin{align*}
s_1(\tau_{\leq 1} p_! \cP^*(K_S)) &= c_1( \ker (d^{0*})) - c_1(E^{2*}) = c_1(E^2) + c_1((\coker d^0)^*), \\
c_1(R^2 p_* \cP) &= c_1(E^2) - c_1(\im d^1) = c_1(E^2) + c_1((E^1 / \ker d^1)^*).
\end{align*}
In the proof of Proposition \ref{chi=0} we saw that $R^1 p_* \cP$ is torsion. Dualizing the short exact sequence
\[
0 \rightarrow R^1 p_* \cP \rightarrow \coker d^0 \rightarrow E^1 / \ker d^1 \rightarrow 0
\]
shows $(\coker d^{0})^* \cong (E^1 / \ker d^1)^*$. This proves the desired result.
\end{proof}

\begin{proof}[Proof of Theorem \ref{main}]
Fix $S,\beta$ such that $p_g(S)=0$ and neither $\beta$ nor $\k-\beta$ satisfies Condition (\ref{strong}) of the introduction. If $\beta(\beta-\k)<0$, the virtual dimensions of $[H_\beta]^{\vir}$ and $[H_{\k-\beta}]^{\vir}$ are zero and we use Proposition \ref{nonredcycle}. Assume $\beta(\beta-\k)\geq 0$. By Proposition \ref{chi=0} this implies $q(S)=1$ and $\beta(\beta-\k)=0$. By Proposition \ref{nonredcycle}, the invariants are zero when point insertions are present ($m>0$). In the case $m=0$, Theorem \ref{mainthm} implies
\begin{align*}
\sfZ_{\beta}^{P}(S) &=  P^{+}_{S}(\beta) \, (q^{1/2} + q^{-1/2})^{2 \beta^2}, \\
\sfZ_{\k-\beta}^{P}(S) &=  P^{+}_{S}(\k-\beta) \, (q^{1/2} + q^{-1/2})^{2 (\k-\beta)^2} = P^{-}_{S}(\beta)  (q^{1/2} + q^{-1/2})^{2(\k-\beta)^2}.
\end{align*}
The result follows from D\"urr-Kabanov-Okonek's wall-crossing formula Theorem \ref{wall-crossing} and a Grothendieck-Riemann-Roch computation giving $s_1(p_! \cP) = \frac{1}{2}[2\beta-\k]$.
\end{proof}

\appendix 

\section{Reduced stable pair invariants} \label{AppA}


Recall from Section \ref{vircycle} that the natural embedding
\[
\Hilb^n(\cC / H_\beta) \hookrightarrow S^{[n]} \times H_\beta
\]
can be realized as the zero locus of a tautological section of the vector bundle $\O(\cC)^{[n]}$ on $S^{[n]} \times H_\beta$ (see \eqref{familytaut}). As we discussed, this induces a relative perfect obstruction theory on $\Hilb^n(\cC / H_\beta)$. We mentioned how the absolute perfect obstruction theory on $H_\beta$ of D\"urr-Kabanov-Okonek was used in \cite{KT1} to construct an absolute perfect obstruction theory on $\Hilb^n(\cC/H_\beta)$.

The Hilbert scheme $H_\beta$ has \emph{another} perfect obstruction theory also originally discovered by D\"urr-Kabanov-Okonek \cite{DKO}. This perfect obstruction theory comes from embedding $H_\beta$ in a compact smooth ambient space as follows. Let $A$ be a sufficiently ample divisor\footnote{It suffices to pick $A$ such that $H^1(L(A)) = H^2(L(A)) = 0$ for all $L \in \Pic^\beta(S)$.} and define $\gamma := [A] + \beta$. Then the Abel-Jacobi map makes $H_{\gamma}:=\Hilb_{\gamma}(S)$ into a projective bundle over the Picard variety $\Pic^\gamma(S)$. In particular, $H_\gamma$ is smooth. Consider the closed embedding
\[
H_\beta \hookrightarrow H_\gamma, \ C \mapsto A \cup C.
\]
A point $D$ lies in the image of this map if and only if it contains $A$, i.e.
\[
s_D|_A = 0 \in H^0(\O_{A}(D)),
\]  
where $s_D$ denotes the section cutting out $D \subset S$. The family version of this goes as follows. Let $\D \rightarrow H_\gamma$ be the universal curve and $\pi : H_\gamma \times S \rightarrow H_\gamma$ projection. Then the \emph{sheaf} 
\begin{equation} \label{F}
F:=\pi_* (\O(\D)|_{H_\gamma \times A})
\end{equation}
has a tautological section with zero locus $H_\beta$. Suppose that $\beta$ satisfies the following condition (Condition (i) of Proposition \ref{chi=0})
\begin{equation} \label{weak}
H^2(L) = 0 \ \mathrm{for \ all \ } \textit{effective} \mathrm{ \ line \ bundles} \  L \ \mathrm{with} \ c_1(L) = \beta.
\end{equation}
Note that this condition is weaker than Condition \eqref{strong} of the introduction. Then $H^1(\O_A(A+C)) = 0$ for any $C \in H_\beta$. By semicontinuity and base change, $R^1 \pi_* (\O(\D)|_{H_\gamma \times A})$ is zero on a Zariski open neighbourhood of $H_\beta$. Hence $F$ is a vector bundle on a Zariski open neighbourhood of $H_\beta$. This construction gives a perfect obstruction theory on $H_\beta$ which we refer to as the \emph{reduced} perfect obstruction theory (this terminology was not used by D\"urr-Kabanov-Okonek). We denote the corresponding virtual cycle by $[H_\beta]^{\red}$. The reduced perfect obstruction theory on $H_\beta$ can be combined with the relative perfect obstruction theory on $\Hilb^n(\cC/H_\beta)$ to give another absolute perfect obstruction theory on $\Hilb^n(\cC/H_\beta)$. This was carried out in \cite[Appendix]{KT2}.  It turns out that the resulting virtual cycle $[\Hilb^n(\cC/H_\beta)]^{\red}$ coincides with the one coming from $\C^*$-localization of \emph{reduced} stable pair theory of $X=K_S$ to the component $\Hilb^n(\cC/H_\beta)$ of the $\C^*$-fixed locus \cite[Appendix]{KT2}. Note that Condition \eqref{weak} is automatic when $p_g(S) = 0$. In this case one can show that $[H_\beta]^{\red} = [H_\beta]^{\vir}$ and $[\Hilb^n(\cC / H_\beta)]^{\red} = [\Hilb^n(\cC / H_\beta)]^{\vir}$.

If $\beta$ satisfies the stronger condition  
\begin{equation*} 
H^2(L) = 0 \ \mathrm{for \ all \ line \ bundles} \  L \ \mathrm{with} \ c_1(L) = \beta,
\end{equation*}
i.e.~Condition \eqref{strong} of the introduction, then $R^1 \pi_* (\O(\D)|_{H_\gamma \times A}) = 0$ on $H_\gamma$ and $F$ is a vector bundle on $H_\gamma$. Denote the embedding 
$$
\Hilb^n(\cC / H_\beta) \hookrightarrow S^{[n]} \times H_\gamma
$$ 
by $\iota$. Similarly to \eqref{familytaut}, define
\begin{equation*} \label{familytaut2}
\O(\D-A)^{[n]} := \pi_* \Big( \O\big(S^{[n]} \times \D - S^{[n]} \times A \times H_\gamma\big) \big|_{\cZ \times H_\gamma} \Big), 
\end{equation*}
where $\pi : S^{[n]} \times S \times H_\gamma \rightarrow S^{[n]} \times  H_\gamma$ denotes projection. When Condition \eqref{strong} holds one can compute the virtual cycle as follows \cite[Thm.~A.7]{KT1}
\begin{equation} \label{vircyclestrong}
\iota_* [\Hilb^n(\cC / H_\beta)]^{\red} = c_r(F) . c_n(\O(\D - A)^{[n]}),
\end{equation}
where $r:=\chi(\beta(A)) - \chi(\beta)$. Here $\chi(\beta)$ is the holomorphic Euler characteristic of curves in $H_\beta$ 
\[
2\chi(\beta) = \beta(\beta -\k) + 2\chi(\O_S)
\]
and $\chi(\beta(A))$ is defined similarly. The virtual dimension of $[\Hilb^n(\cC / H_\beta)]^{\red}$ is
\[
v = \frac{\beta(\beta-\k)}{2}+p_g(S)+n,
\]
which is $p_g(S)$ larger than the virtual dimension of $[\Hilb^n(\cC / H_\beta)]^{\vir}$. 

When only the weaker Condition \eqref{weak} is satisfied we make the following somewhat surprising observation, which is more or less an immediate corollary of \cite[Lem.~2.17]{DKO}. 
\begin{proposition} \label{redcycle}
Fix $S, \beta$ such that Condition (\ref{weak}) is satisfied, $H_\beta \neq \varnothing$, and $\beta(\beta - \k) \geq 0$. Then $F$ is a vector bundle on $H_\gamma$ even though $R^1 \pi_* (\O(\D)|_{H_\gamma \times A})$ is in general non-zero. Consequently
\[
\iota_* [\Hilb^n(\cC / H_\beta)]^{\red} = c_r(F) . c_n(\O(\D - A)^{[n]})
\] 
and its virtual dimension is $v = \frac{\beta(\beta-\k)}{2}+p_g(S)+n$.
\end{proposition}
\begin{proof}
Let $p : \Pic^\beta(S) \times S \rightarrow \Pic^\beta(S)$ be projection and let $\cP$ be a choice of Poincar\'e bundle on $\Pic^\beta(S) \times S$. Let
\[
\E:=[E^0 \stackrel{d^0}{\rightarrow} E^1 \stackrel{d^1}{\rightarrow} E^2]
\] 
be a resolution of $Rp_* \cP$ by locally free sheaves. Then D\"urr-Kabanov-Okonek found out that $\ker d^1$ is locally free (Claim). The reason for Claim is the following. If $H_{\k-\beta} = \varnothing$, then $R^2 p_* \cP = 0$ because $H_{\k - \beta} = \PP(R^2 p_* \cP)$ \cite[Lem.~2.15]{DKO}. In this case $d^1$ is surjective and $\ker d^1$ is locally free. Suppose $H_\beta, H_{\k-\beta}$ are both non-empty. Then we saw in Proposition \ref{chi=0}  and its proof (i.e.~\cite[Lem.~2.17]{DKO} and its proof) that
\begin{align*} 
&R^2 p_* \cP |_U = 0, \ R^0 p_* \cP |_V = 0, \\
& \Pic^\beta(S) = U \cup V, \ \rk Rp_* \cP = 0,
\end{align*}
where $U$, $V$ are the complements of the images of $H_{\k-\beta} \rightarrow \Pic^{\k-\beta}(S) \cong \Pic^\beta(S)$, $H_\beta \rightarrow \Pic^\beta(S)$. We see at once that $R^1 p_* \cP$ is torsion. Moreover $R^1 p_* \cP|_V$ is a subsheaf of $E^1 / \im d^0 |_V$. Also $E^1 / \im d^0 |_V$ is locally free because $(d^0)^*|_V$ is surjective. This implies $R^1 p_* \cP|_V$ is zero. Therefore $\ker d^1 |_V = \im d^0|_V = E^0|_V$ is locally free.
Since we already know $\ker d^1 |_U$ is locally free (because $d^1|_U$ is surjective), this establishes Claim.

Back to the resolution $\E$ of $Rp_* \cP$. Take $\E$ of the following form. Let $[E^1 \stackrel{d^1}{\rightarrow} E^2]$ be a resolution of $R p_* \cP_A(A)$ by locally free sheaves and set $E^0 := p_* \cP(A)$. Note that $p_* \cP(A)$ is locally free by choice of $A$. We define $\E$ by the following diagram of exact triangles
\[
\xymatrix
{
\E \ar@{-->}[r] \ar@{-->}[d] & E^0 \ar@{-->}[r] \ar@{=}[d] & [E^1 \stackrel{d^1}{\rightarrow} E^2] \ar^{\cong}[d] \\
R p_* \cP \ar[r] & p_*\cP(A) \ar[r] & Rp_* \cP_A(A).
}
\]
Here $\cP_A$ is short hand for $\cP|_{H_\beta \times A}$. By Claim, $p_* \cP_A(A) \cong \ker d^1$ is locally free. Next, let $\cQ$ be a choice of Poincar\'e bundle on $\Pic^\gamma(S)$. The Abel-Jacobi map
\[
\AJ : H_\gamma = \PP(p_* \cQ) \rightarrow \Pic^\gamma(S)
\]
is a projective bundle with tautological bundle $\O(1)$. Note that $\cQ(1) \cong \O(\D)$ on $H_\gamma \times S$, therefore \eqref{F}
\[
F = \pi_* \O(\D|_{H_\gamma \times A}) \cong \AJ^*( p_* \cQ_A ) (1).
\]
Since $\Pic^\beta(S) \cong \Pic^\gamma(S)$ sends $p_* \cP_A(A)$ to $p_* \cQ_A$ we indeed see that $F$ is locally free. Finally the proposition states that $R^1 \pi_* (\O(\D)|_{H_\gamma \times A})$ is in general non-zero. This is proved in Remark \ref{failure} below.
\end{proof}

If $S,\beta$ satisfies Condition (\ref{strong}), then 
the invariants $P_{\chi,\beta}^{\red}(S,\tau_0(\pt)^m)$ are calculated in \cite{KT2} in the following sense. Via wedging together and integrating over $S$, the classes $\beta, \k \in H^2(S,\Z)$, and $1 \in H^4(S,\Z)$ give elements
\[
[\beta], [\k] \in \Lambda^2 H^1(S,\Z)^*, \ \mathrm{and} \  [1] \in \Lambda^4 H^1(S,\Z)^*. 
\]
Wedging together any combination produces an element 
\[
\Lambda^a [\beta] \wedge \Lambda^b [\k] \wedge \Lambda^c [1] \in \Lambda^{2q(S)} H^1(S,\Z)^* \cong \Z, \ \mathrm{where} \ a+b+2c=q(S).
\]
Here the canonical isomorphism with $\Z$ comes from choosing any integral basis of $H^1(S,\Z) \subset H^1(S,\R)$ compatible with the orientation coming from the complex structure. We then have:
\begin{theorem} \cite[Thm.~1.2]{KT2} \label{old}
Fixing $q,p_g,m,n$, there exists a universal function $F_{q,p_g,m,n}(\mathbf{x})$ with variables $\mathbf{x}:=(x_1, x_2, x_3, x_4, \{x_{abc}\}_{a+b+2c = q},t)$ such that for any $S$ with $q(S)=q$, $p_g(S)=p_g$, and $\beta \in H^2(S,\Z)$ satisfying Condition (\ref{strong}), $P_{\chi,\beta}^{\red}(S,\tau_0(\pt)^m)$ is equal to
\[
F_{q,p_g,m,n}(\beta^2, \beta.\k, \k^2, c_2(S), \{\Lambda^a [\beta] \wedge \Lambda^b [\k] \wedge \Lambda^c [1]\}_{a+b+2c = q},t),
\]
where $\chi = 1-h+n$ and $2h-2=\beta(\beta+\k)$ is the arithmetic genus of $\beta$.
\end{theorem}

\begin{remark} \label{failure}
Suppose the setting is as in Proposition \ref{redcycle}. We now explain why $R^1 \pi_* (\O(\D)|_{H_\gamma \times A})$ is in general non-zero. In Proposition \ref{redcycle} we show that $F$ is a vector bundle and the reduced virtual cycle is given by \eqref{vircyclestrong} when the weaker Condition \eqref{weak} is satisfied. If $R^1 \pi_* (\O(\D)|_{H_\gamma \times A})$ were zero, then the invariants $P_{\chi,\beta}^{\red}(S,\tau_0(\pt)^m)$ satisfy the same universal formula as Theorem \ref{old} by the calculation of \cite{KT2}. However we show by explicit examples in Appendix \ref{AppB} that some invariants $P_{\chi,\beta}^{\red}(S,\tau_0(\pt)^m)$ do \emph{not} satisfy universality (Examples \ref{example1}, \ref{example2}, \ref{example3}).
\end{remark}
 
\section{Failure of universality: examples} \label{AppB}

In this appendix we show that Theorem \ref{old} does \emph{not} hold for all stable pair invariants $P_{\chi,\beta}^{\red}(S, \tau_0(\pt)^m)$, $P_{\chi,\beta}(S, \tau_0(\pt)^m)$. By Main Theorem \ref{mainthm} it suffices to prove $P_S(\beta)$ is not given by universal functions. We show this on elliptic surfaces with $p_g(S) = 0$ using calculations of D\"urr-Kabanov-Okonek \cite{DKO}.


Let $\pi : S \rightarrow C$ be an elliptic fibration over a curve of genus $g$ \cite[Ch.~V.6]{BPV}. We are only interested in the case $S, C$ are algebraic. The generic fibre $F$ is a smooth elliptic curve and we denote by $m_1 F_1$, $\ldots$, $m_r F_r$ the multiple fibres. The canonical divisor is given by
\begin{equation} \label{K}
K_S = \pi^*D + \sum_{i=1}^{r} (m_i - 1) F_i,
\end{equation}
for some divisor $D$ of degree $2g-2 + \chi(\O_S)$ on $C$ \cite[Cor.~12.3]{BPV}. In this section, we will make frequent use of logarithmic transformations \cite[Ch.~V.13]{BPV}. Given a generic
point $x \in C$, a logarithmic transformation replaces the fibre $F$ over $x$ by a multiple $m F$, $m >1$. The new elliptic fibration $\pi' : S' \rightarrow C$ has fibre $m F$ over $x \in C$ and the restrictions $\pi^{-1}(C \setminus \{x\})$, $\pi^{\prime -1}(C \setminus \{x\})$ are biholomorphic as fibre bundles over $C \setminus \{x\}$. One should not think of a logarithmic transformation as a sort of birational transformation. The topology of $S$ can change and $S$ can even become non-algebraic \cite[Ch.~V.13]{BPV}.

\begin{example} \label{example1}
Let $\PP^1 \subset |\O(3)|$ be a generic pencil of cubics on $\PP^2$ and let $S \rightarrow \PP^1$ be the universal curve. This is a rational elliptic fibration so $q(S)=p_g(S) = 0$ and $K_S = -F$ (Equation (\ref{K})). We take $\beta = 6\k$. Clearly $|6K_S| = \varnothing$ so $P_S(\beta) = 0$. Let $S'$ be obtained from $S$ by replacing one general fibre $F$ by a double fibre $2F_1$ and another by a triple fibre $3F_2$. Then $S'$ is one of the famous Dolgachev surfaces\footnote{The surfaces $S$, $S'$ provide homeomorphic compact simply connected 4-manifolds. S.~K.~Donaldson famously proved their $C^\infty$-structures are different \cite{Don}. One can also establish this by showing their Seiberg-Witten invariants are distinct (see \cite{Moo}).}. The surface $S'$ is known to be algebraic satisfying $q(S')=p_g(S') = 0$ and $K_{S'} = -F + F_1 + 2F_2$ (Equation (\ref{K})). In the Chow group, one has relations $2 F_1 = 3 F_2 = F$ so $\k' = \frac{1}{6} [F]$ in $H^2(S',\Q)$. Taking $\beta' = 6\k'$, we see that $|6K_{S'}| = |F| \neq \varnothing$, whereas $|K_{S'}-6K_{S'}| = |-5K_{S'}| = \varnothing$. Consequently, $P_{S'}(\k' - \beta')=0$. D\"urr-Kabanov-Okonek's wall-crossing formula (Theorem \ref{wall-crossing}) states $P_{S'}(\beta') - P_{S'}(\k'-\beta') = 1$, so $P_{S'}(\beta')=1$. Since the Chern numbers of $S,\beta$ and $S',\beta'$ are the same, this is a counter-example to universality. Note that this does not contradict Theorem \ref{old}. Indeed
 \[
H^2(\O(6K_S)) = H^0(\O(-5K_S)) = H^0(\O(5F)) \neq 0,
\]
so $\beta'$ satisfies Condition (\ref{strong}) but $\beta$ only satisfies the weaker Condition (\ref{weak}). \hfill $\oslash$
\end{example}

In order to find more counter-examples to universality, we use the following result \cite[Prop.~4.8]{DKO} (see also R.~Friedman and J.~W.~Morgan \cite[Prop.~4.4]{FM}).
\begin{proposition}[D\"urr-Kabanov-Okonek] \label{ellfib}
Suppose $\beta \in H^2(S,\Z)$ satisfies $\beta^2 = \beta . [F] = 0$. Then
\[
P_S(\beta) = \sum_{{\scriptsize{\begin{array}{c} d[F] + \sum_i a_i [F_i] = \beta \\ d \geq 0, \ 0 \leq a_i < m_i \end{array}}}} (-1)^d \binom{2g-2+\chi(\O_S)}{d}.
\]
\end{proposition}

Here we should recall the usual conventions on binomial coefficients. For each $b \geq 0$, define 
$$
\binom{a}{b}= \frac{1}{b!}a(a-1)\cdots(a-b+1).
$$ 
In particular, $\binom{a}{b} = 1$ for $b=0$, $\binom{a}{b} = 0$ for $0 \leq a < b$, and $\binom{-a}{b} = (-1)^b \binom{a+b-1}{b}$. 

\begin{example} \label{example2}
Let $S$ be an hyper-elliptic surface and $\beta = d[F]$ for any $d \geq 0$. Note that $q(S)=1$ and $p_g(S)=0$. Proposition \ref{ellfib} implies $P_S(\beta) = 0$ for $d>0$ and $P_S(\beta) = 1$ for $d=0$. Since $\beta^2 = \beta.\k = \k^2 = c_2(S) = [\beta] = [\k] = 0$ for any $d \geq 0$, this also provides a counter-example to universality. Although $K_S$ is a non-trivial torsion element of $A^1(S)$, its class $\k=0 \in H^2(S,\Q)$. The class $\beta = 0$ satisfies Condition (\ref{weak}) but not the stronger Condition (\ref{strong}) since $H^2(\O(K_S)) \neq 0$. Hence, there is no contradiction with Theorem \ref{old}. \hfill $\oslash$
\end{example}

Finally, we apply Proposition \ref{ellfib} to a special class of logarithmic transformations discussed in \cite[Sect.~4.2]{DKO}. They will provide more interesting counter-examples to universality. We recall their construction. Fix an elliptic curve $F = \C / \Gamma$ with lattice $\Gamma = \langle 1, \omega \rangle \subset \C$. We apply logarithmic transformations to $\PP^1 \times F \rightarrow \PP^1$ as follows. Fix a point $t_1 \in \PP^1$ and an $m_1$-torsion point $\zeta_1 \in F$, $m_1>0$. The logarithmic transformation $L_{t_1}(m_1, \zeta_1)(\PP^1 \times F)$ replaces the fibre over $t_1$ by $m_1 F_1$. Continuing in this fashion with other distinct points $t_2, \ldots, t_r \in \PP^1$, one obtains a smooth compact complex surface
\[
S:=L_{\underline{t}}(\underline{m},\underline{\zeta})(\PP^1 \times F),
\]
which is an elliptic fibration over $\PP^1$. It has generic fibre $F$ and multiple fibres $m_1 F_1$, $\ldots$, $m_r F_r$. The following proposition \cite[Sect.~4.2]{DKO} summarizes the relevant geometry.
\begin{proposition}[D\"urr-Kabanov-Okonek] \label{log}
Suppose $\zeta_1, \ldots, \zeta_r$ are of the form $\zeta_i = \frac{u_i + v_i \omega}{m_i}$ for integers $u_i$, $v_i$ satisfying $\gcd(m_i,u_i,v_i) = 1$. 
\begin{enumerate}
\item[(1)] The surface $S$ is projective if and only if $\sum_{i=1}^r \zeta_i = 0$.
\item[(2)] Suppose (1) is satisfied. Then $H^2(S,\Z) \cong \Z \oplus G$, where $G$ is the free abelian group generated by $[F], [F_1], \ldots, [F_r]$ modulo the relations
\begin{align*}
&m_1 [F_1] = \cdots = m_r [F_r] = [F], \\
&u_1 [F_1] + \cdots + u_r [F_r] = 0, \ v_1 [F_1] + \cdots + v_r [F_r] = 0.
\end{align*}
\item[(3)] Suppose (1) is satisfied. Let $\Gamma' \subset \C$ be the lattice generated by the elements $1,\omega, \zeta_1, \ldots, \zeta_r$ and consider the Albanese map $\Alb : S \rightarrow \Alb(S)$. Then there exists an isomorphism $\Alb(S) \cong \C / \Gamma'$ such that the following diagram commutes
\begin{displaymath}
\xymatrix
{
F \ar^{\Alb|_F}[r] \ar@{=}[d] & \Alb(S) \ar^{\cong}[d] \\
\C / \Gamma \ar[r] & \C / \Gamma',
}
\end{displaymath}
where the bottom map is induced by $\Gamma \subset \Gamma'$.
\end{enumerate}
\end{proposition}

The following example is used in \cite[Ex.~4.14]{DKO} to provide a surface $S$ with $p_g(S)=0$ and $P_S(\k) \neq 0$. We use it to give an interesting example where universality fails.
\begin{example} \label{example3}
Take $\zeta_1 = \frac{1+\omega}{3}$, $\zeta_2 = \frac{1}{3}$, $\zeta_3 = \frac{1}{3}$, and $\zeta_4=-\frac{3+\omega}{3}$. By Proposition \ref{log}, $S$ is projective, $[F] = 3[F_1]$, $[F_4] = [F_1]$, $[F_3] = 2[F_1] - [F_2]$, and
\begin{align*}
H^2(S,\Z) &\cong \Z \oplus \langle [F_1], [F_2] \ | \ 3[F_1] = 3[F_2]\rangle_\Z \\
&\cong \Z^{\oplus 2} \oplus \Z / 3 \Z.
\end{align*}
By Equation (\ref{K}), $K_S = 2F_1$ and $\k = \frac{2}{3}[F] \in H^2(S,\Q)$. We fix $\beta = n [F_1] + \epsilon [F_2]$ with $n \in \Z_{\geq 0}$ and $\epsilon = 0,1,2$. The surface $S$ satisfies $q(S)=1$ and $p_g(S)=0$.  Clearly, $\beta^2 = \beta.\k=\k^2=c_2(S)=0$. Let $E$ be the class of the fibre of $S \rightarrow \Alb(S)$. Then Proposition \ref{log} (3) implies $E.F = 9$. Hence $[\beta] = \beta.E = 3(n+\epsilon)$. By Proposition \ref{ellfib}
\[
P_S(\beta) = \sum_{{\scriptsize{\begin{array}{c} (3d+a_1+2a_3+a_4)[F_1]+(a_2-a_3)[F_2] = n[F_1]+\epsilon[F_2] \\ d \geq 0, \ a_i=0,1,2 \end{array}}}} (d+1).
\]
For all $(n,\epsilon) \not\in \{(0,0),(1,0),(2,0),(0,1)\}$, this is equal to
\[
P_S(\beta) = [\beta]-3.
\]
For $(n,\epsilon) = (0,0),(1,0),(2,0),(0,1)$, we get the sporadic values 
$$
P_S(\beta)=1,2,4,1.
$$ 
This gives another counter-example to universality. \hfill $\oslash$
\end{example}

\begin{remark} 
One can consider reduced stable pair invariants with other insertion classes such as 
\begin{equation} \label{invgamma}
P_{\chi,\beta}^{\red}(S,\tau_0(\pt)^m \tau_0(\gamma_1) \ldots \tau_0(\gamma_{2q(S)})),
\end{equation}
where $\gamma_1, \ldots, \gamma_{2q(S)} \in H_1(S)/\mathrm{torsion}$ is an integral oriented basis \cite[Sect.~3]{KT2}. The $H_1$-insertions cut $H_\beta$ down to a linear system $|L| \subset H_\beta$. Fix any $S,\beta$ with $\beta$ satisfying Condition (\ref{weak}) but not necessarily the stronger Condition (\ref{strong}). Suppose $H_\beta \neq \varnothing$ and $\beta(\beta-\k) \geq 0$. Using Proposition \ref{redcycle}, it is easy to see that \cite[Sect.~3]{KT2} continues to hold. Therefore \eqref{invgamma} is given by a universal polynomial in $\beta^2$, $\beta.\k$, $\k^2$, $c_2(S)$ exactly as in \cite[Thm.~1.1]{KT2}. Note that this does not contradict Example \ref{example1} where $|6K_S| = \varnothing$. \hfill $\oslash$
\end{remark}

\begin{remark} \label{(iii)examples}
The conditions for the duality formula of Theorem \ref{main} are: $p_g(S)=0$, $H_\beta$, $H_{\k-\beta}$ are both non-empty, and $\beta(\beta-\k) \geq 0$. Proposition \ref{chi=0} implies $\beta(\beta-\k)=0$, $q(S)=1$, and $S$ is not a ruled surface or a blow-up of a ruled surface. Therefore $S$ is hyper-elliptic, minimal properly elliptic, or a blow-up thereof. Conversely, any hyper-elliptic surface $S$ or blow-up thereof with $\beta=\k$ satisfies the conditions of Theorem \ref{main}. These examples are boring because $P_{S}^{\pm}(\k) = 1$ (by Example \ref{example2} and the blow-up formula \cite[Thm.~3.12]{DKO}). More exciting examples are provided by $S$ as in Proposition \ref{log} and $\beta=\k$. From Theorem \ref{wall-crossing} it follows that these surfaces generally have $H_\k \neq \varnothing$. Blowing up these surfaces, one obtains examples with $H_\k \neq \varnothing$ and $\k^2 \neq 0$. \hfill $\oslash$ 
\end{remark}


\noindent {\tt{m.kool1@uu.nl}} \\
\noindent Mathematical Institute, Utrecht University \\ 
\noindent Budapestlaan 6, 3584 CD Utrecht, The Netherlands 
\end{document}